\documentclass[a4paper,11pt,reqno]{amsart}
\usepackage{latexsym}
\usepackage{amssymb,amsfonts,amsmath,mathrsfs}
\newtheorem{theorem}{Theorem}[section]

\newtheorem{lemma}{Lemma}[section]

\numberwithin{equation}{section}

\begin{document}

\title[On Balazard, Saias, and Yor's equivalence to the Riemann hypothesis]{On Balazard, Saias, and Yor's equivalence to the Riemann Hypothesis}
\author{H. M. Bui, S. J. Lester, and M. B. Milinovich}
\address{Institut f\"{u}r Mathematik, Universit\"{a}t Z\"{u}rich, Z\"{u}rich CH-8057, Switzerland}
\email{hung.bui@math.uzh.ch}
\address{Department of Mathematics, University of Rochester, Rochester, NY 14627 USA}
\email{lester@math.rochester.edu}
\address{Department of Mathematics, University of Mississippi, University, MS 38677 USA}
\email{mbmilino@olemiss.edu}

\allowdisplaybreaks

\maketitle

\begin{abstract}
Balazard, Saias, and Yor proved that the Riemann Hypothesis is equivalent to a certain weighted integral of the logarithm of the Riemann zeta-function along the critical line equaling zero. Assuming the Riemann Hypothesis, we investigate the rate at which a truncated version of this integral tends to zero, answering a question of Borwein, Bradley, and Crandall and disproving a conjecture of the same authors. A simple modification of our techniques gives a new proof of a classical Omega theorem for the function $S(t)$ in the theory of the Riemann zeta-function.
\end{abstract}

\section{Introduction}

Let $\zeta(s)$ denote the Riemann zeta-function. In \cite{BSY}, Balazard, Saias, and Yor gave an elegant proof of the formula
\begin{equation} \label{BSY uncond}
\int_{\Re(s)=1/2} \frac{\log|\zeta(s)|}{|s|^2} |ds| = 2 \pi \sum_{\beta > 1/2} \log\left|\frac{\rho}{1\!-\!\rho}\right|,
\end{equation}
where the sum runs over the nontrivial zeros $\rho=\beta+i\gamma$ of $\zeta(s)$ with real part strictly greater than $1/2$. Since the Riemann Hypothesis (RH) states that $\beta=1/2$ for all the nontrivial zeros of $\zeta(s)$, it follows that RH is equivalent to the expression
\begin{equation} \label{RH1}
\int_{\Re(s)=1/2} \frac{\log|\zeta(s)|}{|s|^2} |ds| = 0.
\end{equation}
This equivalence led Borwein, Bradley, and Crandall \cite{BBC} to study the function
\[
I(T) =\int_{-T}^T \frac{\log|\zeta(\frac{1}{2}\!+\!it)|}{\frac{1}{4}\!+\!t^2} \, dt.
\]
Since by \eqref{RH1}, RH is equivalent to the assertion that $I(T)\to 0$ as $T\to \infty$, they asked the following question: What are the admissible positive values of $\alpha$ such that $I(T) = O(T^{-\alpha})$ as $T\to \infty$ on RH? Based upon numerical evidence, they conjectured that $I(T) = O(T^{-2})$.

In this note, we answer their question and disprove their conjecture by showing that $I(T) = O(T^{-\alpha})$ for any fixed positive $\alpha <2$ as $T \to \infty$, but that $I(T) \ne O(T^{-2})$. Precisely, we prove the following theorem.

\begin{theorem}\label{th1}
Assume RH. Then we have
\begin{equation}\label{est1}
I(T) = O\!\left( \frac{1}{T^2} \frac{\log T}{(\log\log T)^2} \right)
\end{equation}
and
\begin{equation}\label{est2}
I(T) = \Omega\! \left( \frac{1}{T^2} \frac{\sqrt{ \log T }}{(\log\log T)^{3/2}} \right)
\end{equation}
as $T \to \infty$.
\end{theorem}

Additionally, by estimating the tails of the integral
in \eqref{BSY uncond} we obtain an
unconditional formula for $I(T)$ in terms
of the nontrivial zeros of the Riemann zeta-function.
\begin{theorem} \label{th2}
For $T \ge 3$, we have
\begin{equation} \label{I form}
I(T)=2\pi \sum_{\substack{-T \le \gamma \le T \\ \beta >1/2}} 
\log \bigg|\frac{\rho}{1-\rho} \bigg|+O\bigg(\frac{1}{T^2}\log T \bigg). 
\end{equation}
\end{theorem}
Through a straightforward modification of our argument
it can be shown that the Lindel\"of Hypothesis implies
that the error term in \eqref{I form} is $o(T^{-2} \log T)$
as $T \rightarrow \infty$. We remark that the proof of Theorem \ref{th2} does not give a new proof of \eqref{BSY uncond} since
we merely truncate the integral. However, we will show how to adapt the method used to prove Theorem \ref{th2} to give a simple, 
new proof of \eqref{BSY uncond} that relies only on standard techniques in complex analysis.

In the final section, we give a new proof of a classical Omega theorem of Montgomery for the function $S(t)$.

\section{Various lemmas}

Our first two lemmas concern integrals of the logarithm of the Riemann zeta-function (one unconditional and the other conditional upon RH).

\begin{lemma} \label{horizontal}
Uniformly for $1 \leq c \leq 2$ and $t\ge 3$ we have
\[
\int_{1/2}^c \big| \log\zeta(\sigma\!+\!it)\big| \, d\sigma \ll \log t.
\]
\end{lemma}
\begin{proof}
See Lemma $\beta$ of Titchmarsh \cite{T1}. 
\end{proof}

\begin{lemma}\label{vertical}
Assume RH. Then for $t \geq T \ge 3$ we have
$$
\int_T^t \log \big|\zeta(\tfrac12\!+\!iu)\big| \, du \ll \frac{\log t}{(\log \log t)^2}.
$$
\end{lemma}
\begin{proof}
Under the assumption of RH, Cauchy's theorem implies that
\begin{equation} \notag
\int_T^t \log\big|\zeta(\tfrac12\!+\!iu)\big| \, du= - \int_{1/2}^{3/2}\arg \zeta(\sigma\!+\!it)\, d\sigma
+ \int_{1/2}^{3/2}\arg \zeta(\sigma\!+\!iT) \, d\sigma+O(1).
\end{equation}
We will bound the first integral on the right-hand side of this equation.
The second integral can be handled similarly.

Let $\sigma_t=1/2+(\log \log t)^{-1}$ and write
\begin{equation} \label{integral arg}
\int_{1/2}^{3/2} \! \arg\zeta(\sigma\!+\!it) \, d\sigma=I_1+I_2+I_3,
\end{equation}
where $I_1$ is the portion of the integral over $[1/2, \sigma_t)$, $I_2$ is the portion over $[\sigma_t, 3/4)$,
and $I_3$ is the portion over $[3/4, 3/2]$. By Theorem 13.21 of \cite{MV}, we have
$\arg \zeta(\sigma+it)\ll \log t/\log \log t$ for $\sigma \geq 1/2$. Thus,
\[
I_1 \ll \frac{\log t}{(\log \log t)^2}.
\]
For $\sigma_t \le \sigma < 3/4$ it follows from Corollary 13.16 of  \cite{MV}
that
$\arg\zeta(\sigma+it) \ll (\log t)^{(2-2\sigma)}/\log \log t$. Hence
\[
I_2 \ll \frac{\log t}{ (\log \log t)^2}.
\]
Finally, Corollary 13.16  of  \cite{MV} also implies that $\arg\zeta(\sigma+it) \ll (\log t)^{1/2}$ uniformly for $3/4\le \sigma \le 3/2$, and we have
\[
I_3 \ll(\log t)^{1/2}.
\]
The lemma now follows by inserting the estimates for $I_1, I_2$ and $I_3$ into \eqref{integral arg}.
\end{proof}

Next we prove two key lemmas which are used to prove the estimate \eqref{est2} in Theorem \ref{th1}.

\begin{lemma} \label{mvt}
Assume RH. For any sequence of complex numbers $\{r(n)\}$ let 
\[
R(t)=\sum_{n\le N} \frac{r(n)}{n^{it}}.
\]
Then uniformly for $1/2 \le \alpha \le 2$ , $h\in\mathbb{R}$, $N>1$, $T\ge 3$, and $\varepsilon>0$ we have
\begin{equation*}
\int_T^{2T} \log \zeta(\alpha\!+\!it\!+\! ih) \big| R(t) \big|^2dt=T \sum_{mn \leq N} \frac{\Lambda(n) r(m) \overline{r(mn)} }{n^{\alpha+ih} \log n}+O\bigg( N (\log TN)^{3/2+\varepsilon} \sum_{n \leq N}|r(n)|^2\bigg).
\end{equation*}
\end{lemma}
\begin{proof}
Let $c=1+(\log N)^{-1}$, $\mathcal R(s)=\sum_{n \leq N} r(n) \, n^{-s}$, and $\overline{\mathcal R}(s)=\sum_{n \leq N} \overline{r(n)} \, n^{-s}$. We shall consider the case $1/2\leq\alpha\leq c$. The remaining case $c\leq\alpha\leq2$ is treated similarly to $\mathcal{I}_3$ below.

By the elementary inequality $2|ab| \leq |a|^2+|b|^2$ it follows that 
\[
|r(m) r(n)| \Big(\frac{m}{n} \Big)^{\sigma-\alpha} \leq \frac12\Bigg(\frac{|r(m)|^2 \Delta}{n^{2(\sigma-\alpha)}}+ \frac{|r(n)|^2 m^{2(\sigma-\alpha)}}{\Delta}\Bigg)
\]
for any $\Delta>0$. Thus,
\begin{equation*}
\begin{split}
\big| \mathcal R(s-\alpha-ih)\overline{\mathcal{ R}}(\alpha+ih-s)\big| \leq & \sum_{m,n \leq N} |r(m) r(n)| \Big( \frac{m}{n}\Big)^{\sigma-\alpha} \\
\ll &
\Bigg( \Delta \sum_{m \leq N} \frac{1}{m^{2(\sigma-\alpha)}} +\frac{1}{\Delta} \sum_{m \leq N} m^{2(\sigma-\alpha)}\Bigg) \sum_{n \leq N} |r(n)|^2
\\
\ll & \Bigg( \Delta N^{1-2(\sigma-\alpha)} \log N +\frac{N^{1+2(\sigma-\alpha)}}{\Delta} \Bigg) \sum_{n \leq N} |r(n)|^2
\end{split}
\end{equation*}
uniformly for $\alpha \le \sigma \le c$. Choosing $\Delta=N^{2(\sigma-\alpha)}(\log N)^{-1/2}$, we conclude that
\begin{equation} \label{R bd}
\big| \mathcal  R(s-\alpha-ih)\overline{\mathcal{ R}}(\alpha+ih-s)\big|  \ll N (\log N)^{1/2} \sum_{n \leq N} |r(n)|^2
\end{equation}
uniformly for $\alpha \leq \sigma \leq c$.

Let $\mathscr{C}$ be the positively oriented rectangle with vertices at $\alpha+i(T+h)$, $c+i(T+h)$, $c+i(2T+h)$, and $\alpha+i(2T+h)$. We write
\[
i \int_{\mathscr{C}} \log \zeta(s) \, \mathcal R(s-\alpha-ih) \, \overline{\mathcal R}(\alpha+ih-s) \, ds=\mathcal I_1+ \mathcal I_2+ \mathcal I_3+\mathcal I_4,
\]
where $\mathcal I_1$, $\mathcal I_2$, $\mathcal I_3$, $\mathcal I_4$ are the parts of the integral over the left, bottom, right, and top
edges of $\mathscr{C}$, respectively. Cauchy's theorem implies that
\[
\mathcal I_1+\mathcal I_2+\mathcal I_3+\mathcal I_4=0.
\]
Thus, after an obvious variable change, we have 
\begin{equation} \label{start}
\int_T^{2T} \log \zeta(\alpha\!+\!it\!+\! ih) \, \big| R(t) \big|^2 \, dt=-\mathcal I_3+O\big(|\mathcal I_2|\!+\!|\mathcal I_4|\big).
\end{equation}
By \eqref{R bd} and Lemma \ref{horizontal} we have
\begin{equation} \label{two four}
|\mathcal I_2|\!+\! |\mathcal I_4| \ll N (\log NT)^{3/2} \sum_{n \leq N} |r(n)|^2.
\end{equation}
It remains to estimate $\mathcal I_3$.

In $\mathcal I_3$, we express $\log\zeta(s)$ as an absolutely convergent Dirichlet, interchange summation and integration, and then integrate term-by-term to obtain 
\begin{equation} \label{main id}
-I_3= T \sum_{mn \leq N}\frac{\Lambda(n) r(m) \overline{r(mn)} }{n^{\alpha+ih} \log n}+
O\Bigg(\sum_{k=2}^{\infty} \sum_{\substack{m,n\le N\\n \neq km}}  \frac{\Lambda(k)}{k^{c} \log k}  \frac{|r(m)r(n)|}{|\log \frac{n}{km}|}  \Big( \frac{n}{m}\Big)^{c-\alpha}\Bigg).
\end{equation}
To bound the error term, we first note that
\begin{eqnarray*}
\sum_{\substack{m,n \leq N \\n \neq km}}
\frac{|r(m)r(n)|}{|\log \frac{n}{km}|} 
\Big( \frac{n}{m}\Big)^{c-\alpha} 
&\ll&  \Delta \sum_{n \leq N} |r(n)|^2
\sum_{\substack{m \leq N \\ n \neq km}} \frac{1}{m^{2(c-\alpha)} |\log \frac{n}{km} |}\\
&&\qquad\qquad+
\frac{1}{\Delta} \sum_{m \leq N} |r(m)|^2
\sum_{\substack{n \leq N \\ n \neq km}} \frac{n^{2(c-\alpha)}}{|\log \frac{n}{km}|}
\end{eqnarray*}
for any $\Delta>0$. Next, using standard techniques, we have
\[
\sum_{\substack{m \leq N \\ n \neq km}} \frac{1}{m^{2(c-\alpha)} |\log \frac{n}{km} |}  \ll N^{1-2(c-\alpha)} (\log N)^2 \quad \text{and} \quad \sum_{\substack{n \leq N \\ n \neq km}} \frac{n^{2(c-\alpha)}}{|\log \frac{n}{km}|} \ll N^{1+2(c-\alpha)}\log N
\]
uniformly in $k$.
Hence
\[
\sum_{\substack{m,n \leq N \\n \neq km}} \frac{|r(m)r(n)|}{|\log \frac{n}{km}|}  \Big( \frac{n}{m}\Big)^{c-\alpha}  \ll \Bigg( \Delta N^{1-2(c-\alpha)} (\log N)^2 + \frac{N^{1+2(c-\alpha)}\log N}{\Delta} \Bigg) \sum_{n \leq N} |r(n)|^2.
\]
Choosing $\Delta=N^{2(c-\alpha)}(\log N)^{-1/2}$, it follows that the big-$O$ term in \eqref{main id} is 
\[
\ll N (\log N)^{3/2} \sum_{k=2}^{\infty} \frac{\Lambda(k)}{k^{c}\log k} \sum_{n \leq N} |r(n)|^2 \ll N(\log N)^{3/2} \log \log N \sum_{n \leq N} |r(n)|^2.
\]
The lemma now follows from this estimate and \eqref{start}--\eqref{main id}.
\end{proof}

\begin{lemma} \label{resonator}
Let $\mu$ and $\nu$ be fixed non-negative integers, $N>1$, and $h \in [0,(\log\log N)^{-1}]$. Then there exist two real-valued arithmetic functions $r^{\pm}\!(n)$ and a positive constant $C$ \textup{(}depending on $\mu$ and $\nu$\textup{)} such that
\[
\sum_{mn\leq N}\frac{\Lambda(n)\sin^\mu(h\log n)r^+\!(m) r^+\!(mn)}{\sqrt{n}(\log n)^\nu} \Bigg/ \sum_{n\leq N}|r^+\!(n)|^2 \ \ge \ C \, h^\mu (\log N)^{1/2}(\log\log N)^{\mu-\nu+1/2}
\]
and
\[
\sum_{mn\leq N}\frac{\Lambda(n)\sin^\mu(h\log n)r^-\!(m) r^-\!(mn)}{\sqrt{n}(\log n)^\nu} \Bigg/ \sum_{n\leq N}|r^-\!(n)|^2 \ \le \ -C \, h^\mu (\log N)^{1/2}(\log\log N)^{\mu-\nu+1/2}.
\]
\end{lemma}

\begin{proof}
Our proof of this lemma is based upon the ideas in the proof of Theorem 2.1 of Soundararajan \cite{Sound}. We shall prove the first inequality. The second inequality can be proved similarly by choosing $r^-\!(n)=\mu(n)r(n)$, where $\mu(n)$ is the M\"{o}bius function and $r(n)$ is defined below. Throughout the proof, the letter $p$ denotes a prime number.

We choose $r^+\!(n)$ to be the multiplicative function $r(n)$ supported on square-free integers and defined on primes $p$ by
\begin{equation*}
r(p)=\left\{ \begin{array}{ll}
\frac{L(\log p)^{\nu}}{\sqrt{p}}, &\qquad \textrm{if $A<p<B$,}\\
0, & \qquad\textrm{otherwise.} 
\end{array} \right.
\end{equation*}
Here the parameters $A$, $B$ and $L$ are chosen so that
\begin{equation*}\label{200}
A=L^2(\log L)^{2\nu+1},\quad B=L^3,\quad\textrm{and}\quad L^2(\log B)^{2\nu+1}=(2\nu+1)\log N.
\end{equation*}
We note that with our choice we have $r(p)\ll1$,  $L \asymp(\log N)^{1/2}(\log\log N)^{-\nu-1/2}$, and  $\log B < (3/2) \log\log N$, so that $\sin(h \log p) >(h \log p)/2$ for $h \in [0,(\log\log N)^{-1}]$ and $p < B$.
 
With $r^+\!(n)=r(n)$, the denominator on the left-hand side of the first inequality is
\begin{equation*}\label{26}
\sum_{n\leq N}|r(n)|^2\leq \sum_{n=1}^{\infty}r(n)^2=\prod_{p}\big(1+r(p)^2\big).
\end{equation*}
To estimate the numerator, we use Rankin's trick which asserts that for any sequence of non-negative real numbers $\{a_n\}$, and any $\alpha>0$ we have
\begin{displaymath}
\sum_{n> x}a_n\leq x^{-\alpha}\sum_{n>x}a_nn^\alpha\leq x^{-\alpha}\sum_{n=1}^{\infty}a_nn^\alpha.
\end{displaymath}
Therefore,
\begin{eqnarray}\label{7}
&&\sum_{mn\leq N}\frac{\Lambda(n)\sin^\mu(h\log n)r(m) r(mn)}{\sqrt{n} \, (\log n)^\nu}\nonumber\\ 
&&\qquad\qquad\qquad= \sum_{n\leq N}\frac{\Lambda(n)\sin^\mu(h\log n)r(n)}{\sqrt{n} \, (\log n)^\nu}\sum_{\substack{m\leq N/n\\(m,n)=1}}r(m)^2\nonumber\\
&&\qquad\qquad\qquad=\sum_{n\leq N}\frac{\Lambda(n)\sin^\mu(h\log n)r(n)}{\sqrt{n} \, (\log n)^\nu} \prod_{p\nmid n}\big(1+r(p)^2\big)\\
&& \qquad\qquad\qquad\qquad\qquad\qquad+O\Bigg(h^\mu\sum_{n\leq N}\frac{\Lambda(n)r(n)}{\sqrt{n} \, (\log n)^{\nu-\mu}}\Big(\frac{n}{N}\Big)^\alpha\prod_{p\nmid n}\big(1+p^\alpha r(p)^2\big)\Bigg).\nonumber
\end{eqnarray}
Here we have used the inequality $|\sin x| \le x$ for $x \ge 0$ in the big-$O$ term. Note that $r(n)$ is supported on square-free integers, and the inequalities $\sin(h\log p) \gg h \log p$ and $r(p)\ll1$ hold for all $p<B$. 
Using these observations we see that
the ratio of the main term in \eqref{7} to $\sum_{n\le N} |r(n)|^2$
is 
\begin{eqnarray*}
&\gg&\sum_{p\leq N}\frac{\sin^\mu(h\log p) \, r(p)}{\sqrt{p} \, (\log p)^{\nu-1} (1\!+\!r(p)^2)} 
=L \sum_{A<p<B}\frac{\sin^\mu(h\log p)\log p}{p \, (1\!+\!r(p)^2) } \nonumber\\
&\gg& L \sum_{A<p<B}\frac{h^\mu (\log p)^{\mu+1}}{p}= Lh^\mu \, \Bigg( \frac{(\log B)^{\mu+1}}{\mu\!+\!1} - \frac{(\log A)^{\mu+1}}{\mu\!+\!1} + O\big( (\log B )^\mu\big) \Bigg)\nonumber\\
&\gg& Lh^\mu (\log\log N)^{\mu +1} \gg h^\mu (\log N)^{1/2}(\log\log N)^{\mu-\nu+1/2}.
\end{eqnarray*}
On the other hand, the error term in \eqref{7} is
\begin{equation}\label{error in 7}
\begin{split}
&\ll h^\mu L N^{-\alpha} \, \Bigg( \sum_{A<p<B} \frac{(\log p)^{\mu+1}}{p^{1-\alpha} (1+p^\alpha r(p)^2)}  \Bigg) \prod_{p}\big(1+p^\alpha r(p)^2\big)
\\
&\ll \big( 1\!+\!\alpha \log B \big) \, h^\mu L N^{-\alpha} \, \Bigg( \sum_{A<p<B} \frac{(\log p)^{\mu+1}}{p}  \Bigg) \prod_{p}\big(1+p^\alpha r(p)^2\big).
\end{split}
\end{equation}
Note that $B=L^3$ and $L \ll (\log N)^{1/2}$. So by Rankin's trick (with exponent taken to be $1/2$) we have
\begin{equation*}
\begin{split}
\sum_{n\le N} |r(n)|^2&=\sum_{n=1}^{\infty}|r(n)|^2
+O\bigg(\frac{L^2}{N^{1/2}} \sum_{A< p <B} \frac {(\log p)^{2 \nu}}{\sqrt{p}}\bigg)\\
&= \prod_p(1+r(p)^2)+O\bigg( \frac{L^2}{N^{1/2}} B^{3/2} (\log B)^{2 \nu}\bigg)\gg \prod_p(1+r(p)^2).
\end{split}
\end{equation*}
Choosing $\alpha=(\log L)^{-2}$, we see that the ratio of \eqref{error in 7} 
to $\sum_{n\le N} |r(n)|^2 \gg \prod_p(1+r(p)^2)$ is
\begin{equation*}
\begin{split}
&\ll h^\mu L N^{-\alpha} \, (\log B)^{\mu+1}\prod_{p}\Bigg(\frac{1+p^\alpha r(p)^2}{1+r(p)^2}\Bigg)
\\
&\ll    h^\mu L \,  (\log B)^{\mu+1} \exp\Bigg\{ -\alpha \log N + \sum_{A<p<B}(p^\alpha\!-\!1)\frac{L^2(\log p)^{2\nu}}{p}  \Bigg\}
\\
&\ll  h^\mu L \,  (\log B)^{\mu+1}  \exp\Bigg\{ -\alpha\log N+\frac{\alpha L^2}{2\nu\!+\!1}\Big((\log B)^{2\nu+1}-(\log A)^{2\nu+1}\Big)+O\big(\alpha^2L^2(\log B)^{2\nu+2}\big)\Bigg\}
\\
&\ll  h^\mu L \,  (\log B)^{\mu+1}  \exp\Bigg\{ -\frac{1}{2}\frac{\alpha L^2(\log A)^{2\nu+1}}{2\nu\!+\!1}\Bigg\}=o\!\left( h^\mu (\log N)^{1/2} (\log\log N)^{\mu-\nu+1/2} \right)
\end{split}
\end{equation*}
since $ L  (\log B)^{\mu+1} \ll (\log N)^{1/2} (\log\log N)^{\mu-\nu+1/2}$ by our choices of $A, B,$ and $L$. Combining the estimates, the lemma follows.
\end{proof}

\section{Proof of Theorem \ref{th1}}

In this section, we prove Theorem \ref{th1}. Our proof of \eqref{est1} follows from Lemma \ref{vertical}, while our proof of \eqref{est2} is a consequence of the following Omega theorem.

\begin{theorem}\label{log integral} 
Assume RH. Then as $t\to \infty$, we have
\[
\int_{t-h}^{t+h} \log|\zeta(\tfrac{1}{2}\!+\!iu)| \, du = \Omega_\pm\!\left( h \sqrt{ \frac{\log t}{\log\log t}} \right)
\]
uniformly for $h \in[0, (\log\log t)^{-1}].$
\end{theorem}

\begin{proof}
We prove this theorem using Soundararajan's resonance method. 

Let $R(t)=\sum_{n\le N} r(n) n^{-it}$ and observe that
\begin{equation} \label{max}
\max_{T\le t \le 2T} \int_{t-h}^{t+h} \log\big|\zeta(\tfrac{1}{2}\!+\!iu)\big| \, du \ \ge \ \frac{ \int_T^{2T} \Big\{ \int_{t-h}^{t+h} \log|\zeta(\tfrac{1}{2}\!+\!iu)| \, du \Big\} \, |R(t)|^2 \, dt }{ \int_T^{2T} |R(t)|^2 \, dt }
\end{equation}
and
\begin{equation} \label{min}
\min_{T\le t \le 2T} \int_{t-h}^{t+h} \log\big|\zeta(\tfrac{1}{2}\!+\!iu)\big| \, du \ \le \ \frac{ \int_T^{2T} \Big\{ \int_{t-h}^{t+h} \log|\zeta(\tfrac{1}{2}\!+\!iu)| \, du \Big\} \, |R(t)|^2 \, dt }{ \int_T^{2T} |R(t)|^2 \, dt }.
\end{equation}
Making the substitution $u=t+h_1$, using Lemma \ref{mvt} with $\alpha=1/2$, and integrating with respect to $h_1$, the double integral in the numerators in \eqref{max} and \eqref{min} is
\begin{eqnarray} \label{num}
&=&\Re \int_{-h}^h \int_T^{2T} \log \zeta(\tfrac{1}{2}\!+\!it\!+\!ih_1)  |R(t)|^2 \, dt \, dh_1 \nonumber\\
&=& 2T \sum_{mn \le N} \frac{\Lambda(n) r(m) \overline{r(mn)} \sin(h \log n)}{\sqrt{n}( \log n)^2} + O\Bigg( h N (\log TN)^{3/2+\varepsilon} \sum_{n\le N} |r(n)|^2 \Bigg).
\end{eqnarray}
Furthermore, Montgomery and Vaughan's mean-value theorem for Dirichlet polynomials (Corollary 3 of \cite{MV2}) implies that
\begin{equation} \label{mv}
\int_T^{2T} |R(t)|^2 \, dt = \big( T + O(N) \big) \sum_{n\le N} |r(n)|^2.
\end{equation}
Choosing $N=T(\log T)^{-2}$, Lemma \ref{resonator} and equations \eqref{max}--\eqref{mv}  imply that 
\[
\max_{T\le t \le 2T} \int_{t-h}^{t+h} \log|\zeta(\tfrac{1}{2}\!+\!iu)| \, du \ \ge \ c_1 h \sqrt{ \frac{\log T}{\log \log T}}
\]
and
\[
\min_{T\le t \le 2T} \int_{t-h}^{t+h} \log|\zeta(\tfrac{1}{2}\!+\!iu)| \, du \ \le \ -c_2 h \sqrt{ \frac{\log T}{\log \log T}}
\]
uniformly for $h \in [0,(\log \log N)^{-1}]$, where $c_1$ and $c_2$ are (computable) positive constants. The theorem follows. 
\end{proof}

We now prove Theorem \ref{th1}.

\begin{proof}[Proof of Theorem \ref{th1}] We first prove \eqref{est1}. Assuming RH, \eqref{RH1} implies that
\[
\int_{-\infty}^{\infty} \frac{\log |\zeta(\tfrac12\!+\!it)|}{\frac14+t^2} \, dt=0.
\]
Since the integrand is even, it follows that
\[
I(T)=-2\int_T^{\infty} \frac{\log |\zeta(\tfrac12\!+\!it)|}{\frac14+t^2} \, dt.
\]
Integrating by parts and applying Lemma \ref{vertical} we have
\begin{equation} \notag
\begin{split}
I(T) &= -2 \int_T^{\infty} \frac{1}{\frac14+t^2} \, d\bigg(\int_T^t \log |\zeta(\tfrac12\!+\!iu)|  \, du\bigg) 
\\
&= -4 \int_T^{\infty} \frac{t  }{(\frac14+t^2)^2} \bigg(\int_T^t \log |\zeta(\tfrac12\!+\!iu)|  \, du \bigg) \, dt \\
&\ll
\int_T^{\infty}
\frac{1}{t^3} \frac{\log t}{(\log \log t)^2}  \, dt\ll \frac{1}{T^2} \frac{\log T}{ (\log \log T)^2}.
\end{split}
\end{equation}
This completes the proof of  \eqref{est1}.

\smallskip

We now prove  \eqref{est2}. Let $h \in [0,(\log \log t)^{-1}]$ and suppose, for sake of contradiction, that
\begin{equation*}
I(t) =o\Bigg( \frac{1}{t^2} \sqrt{\frac{\log t}{(\log \log t)^3}}\Bigg).
\end{equation*}
Then for $t-h\leq u\leq t+h$ we have
\begin{equation}\label{contra}
I(u)-I(t-h) =o\Bigg( \frac{1}{t^2} \sqrt{\frac{\log t}{(\log \log t)^3}}\Bigg),
\end{equation}
as well. Integrating by parts yields
\begin{equation*}
\begin{split}
\int_{t-h}^{t+h} \log\big |\zeta(\tfrac12\!+\!iu)\big| du= & \int_{t-h}^{t+h}
\big(\tfrac14+u^2\big) \, d\bigg(\int_{t-h}^{u}\frac{\log |\zeta(\tfrac12\!+\!iv)|}{\tfrac14+v^2}  \,  dv\bigg) \\
=&\Big(\tfrac14+(t+h)^2 \Big) \int_{t-h}^{t+h} \frac{\log |\zeta(\tfrac12\!+\!iv)|}{\frac14+v^2}  \,  dv\\
&\qquad \qquad  -\int_{t-h}^{t+h} 2u \, \bigg(\int_{t-h}^u \frac{\log |\zeta(\tfrac12\!+\!iv)|}{\frac14+v^2}  \, dv \bigg)  \, du.
\end{split}
\end{equation*}
Using the assumption \eqref{contra} twice, it follows that
\[
\int_{t-h}^{t+h} \log |\zeta(\tfrac12\!+\!iu)| \, du =o\Bigg(\sqrt{\frac{\log t}{(\log \log t)^3}} \Bigg).
\]
If $h=(\log\log t)^{-1}$, this contradicts Theorem \ref{log integral}, and thus proves \eqref{est2}.
\end{proof}

\section{Proof of Theorem \ref{th2}}
In this section, we use contour integration to prove 
Theorem \ref{th2}. We also show how this method can be modified to give
a new proof of \eqref{BSY uncond} that relies solely on standard techniques from complex analysis.
\begin{proof}[Proof of Theorem \ref{th2}]
First, suppose that $T$
is not an ordinate of a zero of $\zeta(s)$ and
consider
\[
\frac1i \int_{\frac12+iT }^{\frac12+i\infty}
\frac{\log \zeta(s)}{s(1-s)} \, ds.
\]
Let $\mathcal S$ be subset of the region $\sigma>1/2$ and $t>T$,
that excludes all the horizontal segments $1/2+i\gamma$ to $\beta+i\gamma$. 
It follows that
$\log \zeta(s)$ is a single-valued analytic function in $\mathcal S$. Moreover, along each branch cut from
$1/2+i\gamma$ to $\beta+i\gamma$ the values of 
$\log \zeta(s)$ on the upper and lower cuts differ by 
$2\pi i$. Therefore, moving
the contour in the above integral from $\Re(s)=1/2$ to $\Re(s)=\infty$ yields
\begin{equation} \label{contour}
\begin{split}
\frac1i \int_{\frac12+iT }^{\frac12+i\infty}
\frac{\log \zeta(s)}{s(1-s)} \, ds&=2 \pi \sum_{\substack{\gamma>T \\ \beta>1/2}} \int_{\frac12+i\gamma}^{\beta+i\gamma} \frac{1}{s(1-s)} \, ds+\frac{1}{i}
\int_{\frac12+iT}^{\infty+iT} \frac{\log \zeta(s)}{s(1-s)} \, ds.
\end{split}
\end{equation}
Also, we have
\begin{equation} \label{sum}
\begin{split}
 \int_{\frac12+i\gamma}^{\beta+i\gamma} \frac{1}{s(1-s)} \, ds
=\log(\rho)-\log(\tfrac12+i\gamma) 
-\log(1-\rho)+\log(\tfrac12-i\gamma).
\end{split}
\end{equation}

For $\sigma \geq 2$ we have $\log \zeta(s) \ll 2^{-\sigma}$
uniformly in $t$. From this and Lemma \ref{horizontal} it follows that
\[
\int_{\frac12+iT}^{\infty+iT} \frac{\log \zeta(s)}{s(1-s)} \, ds
\ll \frac{1}{T^2}\bigg(
\int_{\frac12}^{2}+\int_{2}^{\infty}\bigg) |\log \zeta(\sigma+iT)| \, d\sigma
\ll \frac{1}{T^2}\big(\log T+1\big).
\]
Taking the real parts in \eqref{contour}, and using the above estimate and \eqref{sum}, we deduce that
\[
\int_{T}^{\infty} \frac{\log |\zeta(\tfrac12+it)|}{\frac14+t^2} \, dt
=2\pi \sum_{\substack{\gamma>T \\ \beta>1/2}} 
\log \bigg|\frac{\rho}{1-\rho}\bigg|+O\bigg(\frac{1}{T^2} \log T \bigg).
\]
Similarly, it can be shown that
\[
\int_{-\infty}^{-T} \frac{\log |\zeta(\tfrac12+it)|}{\frac14+t^2} dt
=2\pi \sum_{\substack{\gamma<-T \\ \beta>1/2}}  \log \bigg|\frac{\rho}{1-\rho}\bigg|+O\bigg(\frac{1}{T^2} \log T \bigg).
\]
 Combining these two estimates and then 
 differencing the resulting formula with
 \eqref{BSY uncond}
completes the proof of the theorem in the case when $T \neq \gamma$. If $T=\gamma$, we note that 
for all sufficiently small
$\varepsilon>0$ the estimate in \eqref{I form} holds for $T=\gamma+\varepsilon$. The theorem now follows in this case by
letting $\varepsilon \rightarrow 0^+$.
\end{proof}

\begin{proof}[Proof of \eqref{BSY uncond}]
Consider the integral
\[
\frac1i \int_{\frac12-i\infty }^{\frac12+i\infty}
\frac{\log ((s-1)\zeta(s))}{s(1-s)} \, ds.
\]
Arguing as in the previous proof, we
move the contour from $\Re(s)=1/2$
to $\Re(s)=\infty$ and deduce that
\begin{equation} \label{contour 2}
\begin{split}
\frac1i \int_{\frac12-i\infty }^{\frac12+i\infty}
\frac{\log ((s-1)\zeta(s))}{s(1-s)} \, ds&=2 \pi \sum_{ \beta>1/2} \int_{\frac12+i\gamma}^{\beta+i\gamma} \frac{1}{s(1-s)} \, ds+
\frac{1}{i}\int_{\mathcal C} \frac{\log ((s-1)\zeta(s))}{s(1-s)} \, ds,
\end{split}
\end{equation}
where $\mathcal C$ is the positively oriented circle centered at $s=1$ with
radius $1/4$. By the calculus of residues  and the 
fact that $\lim_{s \rightarrow 1} ((s-1)\zeta(s))=1$ the last
integral equals zero.
Thus, 
by this and \eqref{sum}, 
taking the real parts in \eqref{contour 2} gives
\[
\int_{-\infty }^{\infty}
\frac{\log |(-\tfrac12+it)\zeta(\tfrac12+it)|}{\frac14+t^2} \, dt
=2 \pi \sum_{ \beta>1/2} \log \bigg|\frac{\rho}{1-\rho} \bigg|.
\]
Note that by residue calculus (or otherwise) we have
\[
\int_{-\infty }^{\infty} \frac{\log |-\tfrac12+it|}{\frac14+t^2}\, dt = \frac{1}{2}\int_{-\infty }^{\infty} \frac{\log(\frac{1}{4}+t^2)}{\frac14+t^2}\, dt =0.
\]
This completes the proof.
\end{proof}

\section{Montgomery's Omega theorem for $S(t)$}

Let $N(t)$ denote the number of non-trivial zeros $\rho=\beta+i\gamma$ of the Riemann zeta-function with $0<\gamma\le t.$ It is well-known that 
\[
N(t) = \frac{t}{2\pi} \log \frac{t}{2\pi} -\frac{t}{2\pi} + \frac{7}{8} + S(t) + O\Big(\frac{1}{t}\Big)
\]
for $t\geq10$. Here, if $t$ is not equal to an ordinate of a zero of $\zeta(s)$, the function $S(t)$ is defined by 
\[
S(t)=\frac{1}{\pi}\Im \log\zeta\big(\tfrac{1}{2}+it\big),
\] 
where the branch of logarithm is obtained by continuous variation along the line segments joining the points $2, 2 + it$, and $\frac{1}{2} + it$, starting with $\arg\zeta(2)=0$. If $t$ corresponds to an ordinate of a zero of $\zeta(s)$ we set
\[
 S(t) = \frac{1}{2} \ \! \lim_{\varepsilon\to 0} \big\{ S(t\!+\! \varepsilon)\!+\!S(t\!-\! \varepsilon) \big\}.
 \]
Assuming RH, it is known that
\[
\big|S(t)\big| \le \Big(\tfrac{1}{4} + o(1) \Big) \frac{\log t}{\log \log t}
\]
as $t\to \infty$ \cite{CCM}. In this section, we illustrate how Lemmas \ref{mvt} and \ref{resonator}  in \textsection 3 can be used to give a new proof of Montgomery's result \cite{M} that
\begin{equation} \label{omega}
S(t) = \Omega_\pm\!\left( \sqrt{\frac{\log t}{\log\log t}} \right)
\end{equation}
assuming RH.  Tsang \cite{Tsang} gave an alternate proof of \eqref{omega}. In contrast to the proofs of Montgomery and Tsang, our proof uses the resonance method. 

\begin{proof}[Proof of \eqref{omega}] Define the auxiliary function
\[
S_1(t) = \int_0^t S(u) \, du
\]
and note that
\begin{equation}\label{S1}
\max_{t\le u \le t+h} \pm S(u) \ge \frac{1}{h} \int_t^{t+h}\!\! \pm S(u) \ du = \frac{\pm\big( S_1(t\!+\!h)-S_1(t) \big)}{h}.  
\end{equation}

We use a result of Littlewood (see Theorem 3 of \cite{L}  or Theorem 9.9 of \cite{T}) that
\[
S_1(t) = \frac{1}{\pi} \int_{1/2}^{2} \log|\zeta(\sigma\!+\!it)| \, d\sigma+ O\big(1\big).
\]
Now taking the real part of the integral in Lemma \ref{mvt}, and integrating with respect to $\alpha$ from $1/2$ to $2$ yields
\begin{equation*}
\begin{split}
\int_T^{2T}S_1(t\!+\!h) |R(t)|^2 \, dt &= \frac{T}{\pi} \sum_{mn \leq N} \frac{\Lambda(n) r(m) \overline{r(mn)} }{ \sqrt{n} \, (\log n)^2} \cos\!\big(h\log n\big)+O\bigg( T\sum_{n \leq N}|r(n)|^2\bigg)
\\
& \qquad + O\bigg( N (\log TN)^{3/2} \sum_{n \leq N}|r(n)|^2\bigg) + O\bigg(\int_T^{2T} |R(t)|^2 \, dt \bigg). 
\end{split}
\end{equation*}
Choosing $N=T(\log T)^{-2}$ and noting that
\[
\int_T^{2T} |R(t)|^2 \, dt = \Big(T+O(N) \Big) \sum_{n \leq N}|r(n)|^2,
\]
we obtain
\[
\frac{\pm\int_T^{2T} \big( S_1(t\!+\!h)-S_1(t) \big) \, |R(t)|^2 \, dt }{\int_T^{2T} |R(t)|^2 \, dt} = \mp\frac{2}{\pi} \frac{\sum_{mn \leq N} \frac{\Lambda(n) r(m) \overline{r(mn)} }{ \sqrt{n} \, (\log n)^2} \sin^2\!\big(\tfrac{h}{2}\log n\big)}{\sum_{n \leq N}|r(n)|^2} + O\big(1\big).
\]
Using Lemma \ref{resonator} with $\mu=\nu=2$ to estimate the ratio of sums on the right-hand side of the above expression, we deduce that
\[
\max_{T\le t \le 2T} \pm \big( S_1(t\!+\!h)-S_1(t) \big) \gg h^2 \sqrt{ \log T \log\log T}
\]
uniformly for $h\in [0,(\log\log N)^{-1}]$. Combining this inequality with the observation in \eqref{S1} and choosing $h=(\log\log N)^{-1}$, the estimate \eqref{omega} follows. 
\end{proof}

We remark that using the resonance method in a different way, the estimate in \eqref{omega} can be refined.  In \cite{BMR}, assuming RH, it is shown that
\[ \max_{T\le t \le 2T} S(t) \ge \frac{1}{\pi} \sqrt{\frac{\log t}{\log\log t}} + O\!\left( \frac{\sqrt{\log t}}{\log\log t} \right) \]
and
\[ \min_{T\le t \le 2T} S(t) \le -\frac{1}{\pi} \sqrt{\frac{\log t}{\log\log t}} + O\!\left( \frac{\sqrt{\log t}}{\log\log t} \right). \]
These are conditional analogues of Soundararajan's unconditional Omega theorem for $\big|\zeta(\tfrac{1}{2}+it)\big|$ in \cite{Sound}. It does not seem, however, that the method in \cite{BMR} can be modified to prove Theorem \ref{log integral}.

\section*{Acknowledgments}
The authors learned of the question and conjecture of Borwein, Bradley, and Crandall from \cite{CP}. Part of this research was completed while the third author was visiting the Institut f\"{u}r Mathematik at the Universit\"{a}t Z\"{u}rich. He would like to thank the Institute, and especially Ashkan Nikeghbali, for the invitation.


\begin{thebibliography}{99}

\bibitem{BSY} M. Balazard, E. Saias, and M. Yor, {\it Notes sur la fonction $\zeta$ de Riemann. II.} (French) Adv. Math. 143 (1999), 284--287.

\bibitem{BBC} J. M. Borwen, D. M. Bradley, and R. E. Crandall, {\it Computational strategies for the Riemann zeta function}, Numerical analysis in the 20th century, Vol. I, Approximation theory. J. Comput. Appl. Math. 121 (2000), 247--296. 

\bibitem{BMR} H. M. Bui, M. B. Milinovich, and M. Radziwi\l\l, {\it Extreme values of $S(t)$ in the theory of the Riemann zeta-function}, preprint.

\bibitem{CCM} E. Carneiro, V. Chandee, and M. B. Milinovich, {\it Bounding $S(t)$ and $S_1(t)$ on the Riemann hypothesis}, to appear in Math. Ann. (2013), DOI
10.1007/s00208-012-0876-z

\bibitem{CP} R. E. Crandall and C. Pomerance, {\it Prime numbers. A computational perspective}, Second edition. Springer, New York, 2005.

\bibitem{L} J. E. Littlewood, \textit{On the zeros of Riemann's zeta-function}, Proc. Camb. Phil. Soc. \textbf{22} (1924), 295--318.

\bibitem{M} H. L. Montgomery, {\it Extreme values of the Riemann zeta function}, Comment. Math. Helv. 52 (1977), 511--518.

\bibitem{MV2} H. L. Montgomery and R. C. Vaughan, {\it Hilbert's inequality}, J. London Math. Soc. 8 (1974), 73--82.

\bibitem{MV} H. L. Montgomery and R. C. Vaughan, {\it Multiplicative Number Theory I. Classical Theory}, Cambridge
Studies in Advanced Mathematics 97 (Cambridge University Press, Cambridge, 2007).

\bibitem{Sound} K. Soundararajan, {\it Extreme values of zeta and L-functions}, Math. Ann. 342 (2008), 467--486.

\bibitem{T1} E. C. Titchmarsh, {\it On the remainder in the formula for $N(T)$, the number of zeros of $\zeta(s)$ in the strip $0<t<T$}, Proc. London Math. Soc. 2 (1927), 247--254. 

\bibitem{T} E. C. Titichmarsh, {\it The theory of the Riemann zeta-function}, Second edition. Edited and with a preface by D. R. Heath-Brown. The Clarendon Press, Oxford University Press, New York, 1986. x+412 pp.

\bibitem{Tsang} K.-M. Tsang, {\it Some $\Omega$-theorems for the Riemann zeta-function}, Acta Arith. 46 (1986), 369--395.

\end{thebibliography}
\end{document}